\documentclass[11pt,bezier]{article}
\usepackage{amsmath,amssymb,amsfonts,euscript,graphicx}
\usepackage{mathtools}
\usepackage{amsfonts}
\usepackage[english]{babel}
\usepackage[all]{xy}
\usepackage{tikz-cd}

\newtheorem{preproof}{{\bf \indent Proof.}}

\newenvironment{proof}[1]{\begin{preproof}{\rm
               #1}\hfill{$\Box$}}{\end{preproof}}

\newtheorem{prop}{\bf\indent Proposition}[section]
\newtheorem{defn}[prop]{\bf\indent Definition}
\newtheorem{cor}[prop]{\bf\indent Corollary}
\newtheorem{example}[prop]{\bf\indent Example}
\newtheorem{thm}[prop]{\bf\indent Theorem}

\newtheorem{lem}[prop]{\bf\indent Lemma}
\title{\bf  \large Commutative rings whose proper ideals \\ are direct sum of cyclically presented modules\thanks
{{\it Key Words}:  Cyclic modules, Cyclically presented modules, Semi-local rings, Principal ideal rings, Prime ideals.} \thanks
{\indent{~~2010 {\it Mathematics Subject Classification}: 13C05, 13F10, 13H99.}}}
\author{{\normalsize  {\sc R. Nikandish${}^{\mathsf{a}}$\thanks{Corresponding author}, {\sc M.J. Nikmehr${}^{\mathsf{b}}$} and {\sc A. Yassine${}^{\mathsf{b}}$}  }
}\vspace{3mm}\\
{\footnotesize{}}\\
{\footnotesize{}}\\
{\footnotesize{${}^{\mathsf{a}}$\it Department of Mathematics, Jundi-Shapur University of Technology,}}\\
{\footnotesize{\rm P.O. BOX \rm{64615-334},
Dezful, Iran}}\\
{\footnotesize{ $\mathsf{r.nikandish@ipm.ir}$}}\\
{\footnotesize{${}^{\mathsf{b}}$\it Faculty of Mathematics, K.N. Toosi
University of Technology, }}\\
{\footnotesize{\rm P.O. BOX \rm{16315-1618}, Tehran, Iran}}\\
{\footnotesize{ $\mathsf{nikmehr@kntu.ac.ir}$}}\quad\quad
{\footnotesize{$\mathsf{yassine\_ali@email.kntu.ac.ir}$}}\\
{\footnotesize{$\mathsf{}$ }}}
\date{}

\begin{document}

\maketitle
\begin{abstract}
{A famous result due to I. M. Isaacs states that if a commutative ring $R$ has the property that every prime ideal is principal, then every ideal of $R$ is principal. This motivates ring theorists to study commutative rings for which every ideal is a direct sum of cyclically presented modules. In this paper, we study commutative rings whose ideals are direct sum of cyclically presented modules.
}
\end{abstract}
%%%%%%%%%%%%%%%%%%%%%%%%%%%%%%%%%%%%%%%%%%%%%%%%%%%%%%%%%%%%%%%%%%%%%%%%%%%%%%%%%%%%%%%%%%%%%%%%%%%%%%%%%%%%%%%%%%%%%%%%%%%%%%%%%%%%%%%%%%%%%%%%%%%%%%%%
\begin{center}{\section{Introduction
}}\end{center}
\par
We assume throughout this paper that all rings are commutative with identity. Let $R$ be a ring and $X$ be a nonempty subset of $R$. The
set of prime ideals of $R$, the set of maximal ideals of $R$, the set of nilpotent elements of $R$, the set of zero-divisors of $R$, the set of integers, the annihilator of $X$ and integers modulo $n$ are denoted by Spec$(R)$, Max$(R)$, Nil$(R)$, $Z(R)$, $\mathbb{Z}$, Ann$(X) = \{a \in R | aX = 0\}$ and $\mathbb{Z}_n$, respectively. By a proper ideal $I$ of $R$, we mean an ideal with $I\neq R$.  A ring $R$ is called \textit {local {\rm(}resp. semilocal{\rm)}} if $R$ has a unique maximal ideal (resp. $R$ has a finite number of maximal ideals). For any undefined notation or terminology in commutative ring theory, we refer the reader to \cite{sha}.

Prime and maximal ideals play a central role in commutative ring theory and so these notions have been studied to specify the structure of all ideals of the rings. We note that two theorems from commutative algebra due to I. M. Isaacs and I. S. Cohen state that, to check whether every ideal in a ring is cyclic (resp. finitely generated), it suffices to test only the prime ideals (see  \cite[p. 8, Exercise 10]{Kaplansky} and \cite[Theorem 2]{Cohen}). It was shown by Kaplansky \cite{Kaplansky2} that, for a commutative Noetherian ring $R$, every maximal ideal of $R$ is principal if and only if every ideal of $R$ is principal. The study of rings over which modules are direct sums of cyclic modules has has been done by several authors. It was shown by K\"othe \cite{Kothe} and Cohen-Kaplansky \cite{CohenKaplansky} that ``a commutative ring $R$ has the property that every module is a direct sum of cyclic modules if and only if $R$ is an Artinian principal ideal ring".  The study of commutative rings $R$ such that the ideals of $R$ are direct sums of  cyclic modules was initiated by Behboodi et al. in \cite{Ghorbani,Heidari} and \cite{Behboodi2}. Later, Moradzadeh-Dehkordi and Shojaee \cite{Moradzadeh} investigated rings $R$ whose ideals  are direct sums of cyclically presented modules. Indeed, they proved the following theorem.
\begin{thm}
{\rm (\cite[Theorem 2.16]{Moradzadeh})} Let $(R, M)$ be a local ring. Then the following statements are equivalent:

$(1)$ Every ideal of $R$ is a direct sum of cyclically presented $R$-modules.

$(2)$ Every ideal of $R$ is a direct summand of a direct sum of cyclically presented $R$-modules.

$(3)$ Every prime ideal of $R$ is a direct sum of cyclically presented R-modules.

$(4)$ Every prime ideal of $R$ is a direct summand of a direct sum of cyclically presented R-modules.

$(5)$ Either $R$ is a principal ideal ring or $M = Rx\oplus Ry$, where $x,y\notin$Nil$(R)$, $R/{\rm Ann}(x)$ and $R/{\rm Ann}(y)$ are principal ideal rings.
\end{thm}
So a natural question is posed: ``What is the class of semi-local rings $R$ for which every ideal of $R$ is a direct sum of cyclically presented $R$-modules". This paper is in this theme and it is devoted to study semi-local rings in which every proper ideal is a direct sum of cyclically presented $R$-modules. It is shown that if $R$ is a semi-local ring with maximal ideals $\{M_1, \ldots,$ $M_t\}$ and $x\in R$, then the cyclic $R$-module $Rx$ is cyclically presented if and only if Ann$(x)$ is a principal ideal of $R$ (Proposition \ref{2.3}). We give a condition $(*)$ (Definition \ref{*}) and it is shown that if $R$ is a  semi-local ring that satisfies condition $(*)$ with maximal ideals $\{M_1, \ldots, M_t\}$ such that every prime ideal $P$ has the form $\bigoplus_{i\in I}Rx_i$ where $Rx_i$ is a direct sum of cyclically presented $R$-modules and $x_i\notin {\rm Nil}(R)$ for every $i\in I$ where $I$ is an index set, then for every maximal ideal $M_j$ of $R$ either $M_j$ is cyclic or $M_j = Rx_{1,j}\oplus Rx_{2,j}$ such that $Rx_{1,j}$, $Rx_{2,j}$ are cyclically presented and $R/{\rm Ann}(x_{1,j})$, $R/{\rm Ann}(x_{2,j})$ are principal ideal rings for every $j\in \{1, \ldots , t\}$ (Theorem \ref{3437}).  It is proved that to check whether every ideal of a  semi-local ring $R$ that satisfies condition $(*)$ is a direct sum of cyclically presented modules, it suffices to test only the prime ideals or the structures of the maximal ideals of $R$ (Theorem \ref{2.112}).

%%%%%%%%%%%%%%%%%%%%%%%%%%%%%%%%%%%%%%%%%%%%%%%%%%%%%%%%%%%%%%%%%%%%%%%%%%%%%%%%%%%%%%%%%%%%%%%%%%%%%%%%%%%%%%%%%%%%%%%%%%%%%%%%%%%%%%%%%%%%%%%%%%%%
\vspace*{1cm}
\begin{center}{\section{Main results
}}\end{center}

Recently in \cite{Moradzadeh}, it has been given two criteria
to check whether every ideal of a commutative local ring $R$ is a direct sum of
cyclically presented modules. They described the ideal structure
of commutative local rings for which every ideal of $R$ is a direct sum of cyclically
presented $R$-modules. This turns out to be a key result to progress
in the proof of commutative semi-local rings. The other important  goal of this paper is to show that commutative semi-local rings whose ideals are direct sums of cyclically presented modules are Noetherian. As a consequence, we describe the ideal structure of commutative semi-local rings in which every ideal of $R$ is a direct sum of cyclically presented $R$-modules. In light of \cite[Proposition 3.4]{Behboodi2} and Nakayama's Lemma,
 the following definition is given.

\begin{defn} \label{*}
Let $R$ be a semi-local ring and Max$(R)=\{M_1, \ldots,$ $M_t\}$. We say that $R$ satisfies condition $(*)$ over $M_i$'s, if the following conditions hold:

$(1)$ If $M_i = Rx_i \oplus L$ for some ideal $L$
of $R$ and some $x \in R$ such that $R/{\rm Ann}(x_i)$ is a principal ideal ring, then
every nonzero element of $Rx_i$ is of the form $ax^n$ for some unit element $a\in R$ and positive integer $n$.

$(2)$ If $M_j = Rx_{1,j} \oplus \cdots \oplus Rx_{n,j}$, then Ann$(x_{i,j}^{\alpha}) \subseteq M_j$ for every $i\in \{1, \ldots, n\}$, $j\in \{1, \ldots, t\}$ and positive integer $\alpha$. Also, for every finitely generated module $N$ over the ring $R$, if $I$ is an ideal of $R$ such that $I \subseteq M_j$ for some $j\in \{1, \ldots, t\}$ and $N = IN$, then $N = 0$.
\end{defn}

\begin{thm} \label{222}
Let $R$ be a semi-local ring such that satisfies condition $(*)$, Max$(R)=\{M_1, \ldots,$ $M_t\}$ and every prime ideal is a direct sum of cyclically presented $R$-modules. If $M_i = Rx_i$, then $R$ is a principal ideal domain and Spec$(R) = \{(0), M_1, \ldots,$ $M_t\}$. In particular, every nontrivial ideal of $R$ is of the form $Rx_i^n$ for some positive integer $n$ and $i\in \{1, \ldots, t\}$.
\end{thm}
\begin{proof}
{ Suppose that $M_i = Rx_i$ such that $x_i \notin$ Nil$(R)$ for every $i\in \{1, \ldots, t\}$. It is clear that if $R$ has at least two distinct maximal ideals $M_1 = Rx_1$ and
$M_2 = Rx_2$, then $x_1,x_2 \notin$ Nil$(R)$, if not, $R$ is a local ring. Then Nil$(R) \neq M_i$ for each $i$. So suppose that $P \in$ Spec$(R) \setminus \{M_1, \ldots, M_t\}$. Then there exists $i\in \{1, \ldots, t\}$ such that $P\subseteq M_i = Rx_i$, and so $P = P \cap Rx_i$ and $x_i\notin P$, as $P \notin \{M_1, \ldots, M_t\}$. Since $P$ is prime, we conclude that $P = Px_i = PRx_i.$  Hence, $P = (0)$, and therefore Spec$(R) = \{(0), M_1, \ldots, M_t\}$. Thus, by \cite[p. 8, Exercise 10]{Kaplansky}, $R$ is a principal ideal ring and since $(0)$ is a prime ideal, $R$ is also a principal ideal domain.

For the ``in particular" statement suppose that $R$ satisfies condition $(*)$ and $I$ is a nontrivial ideal of $R$. Since $R$ is a principal ideal domain, there exists an element $y\in R$ such that $I = Ry$. Assume that $I\subseteq M_i$ for some $i\in \{1, \ldots, t\}$. Hence,
$y = ax_i^n$ for some unit element $a$ of $R$ and positive integer $n$, and thus $I = Rx^n$, for some positive integer $n$.~}
\end{proof}

%\begin{remark}
%{\rm Let $R$ be a semi-local ring such that satisfies condition $(*)$, Max$(R)=\{M_1, \ldots, M_t\}$ and every prime ideal is a direct sum of cyclically presented $R$-modules. If $M_i = Rx_i$, where $x_i \notin$ Nil$(R)$ for some $i\in \{1, \ldots, t\}$, then by the proof of Theorem \ref{222},  Spec$(R) = \{(0), M_1, \ldots, M_t\}$. If $M_i = Rx_i$ such that $x_i^{k_i} = 0$ for some $i\in \{1, \ldots, t\}$ and positive integer $k_i$, then $(0)$ is not a prime, i.e., Spec$(R) = \{M_1, \ldots, M_t\}$ and so $R$ is Artinian.
%}
%\end{remark}

In the following proposition, we show that if $R$ is a semi-local ring and $a\in R$, then the cyclic $R$-module $Ra$ is cyclically presented if and only if Ann$(a)$ is a principal ideal of $R$.

\begin{prop} \label{2.3}
Let $R$ be a semi-local ring, Max$(R)=\{M_1, \ldots,$ $M_t\}$ and $x\in R$. Then the cyclic $R$-module $Rx$ is cyclically presented if and only if Ann$(x)$ is a principal ideal of $R$.
\end{prop}
\begin{proof}
{Let $R$ be a semi-local ring and Max$(R)=\{M_1, \ldots,$ $M_t\}$ and $x\in R$ such that the cyclic $R$-module $Rx$ is cyclically presented. Then $Rx \cong R/I$ for some principal ideal $I$ of $R$. Now, we consider the following diagram:
\[ \xymatrix{
  0 \ar[r] & {\rm Ann}(x) \ar[d]_-{\alpha} \ar[r]^-{f} & R \ar[d]_-{||} \ar[r]^-{g} & R/{\rm Ann}(x) \ar[d]^-{\wr} \ar[r] & 0 \\
  0 \ar[r] & I \ar[r]_-{f'} & R \ar[r]_-{g'} & R/I \ar[r] & 0.
} \]
Therefore, by \cite[Lemma 5.1]{Lam}, we conclude that $R\oplus {\rm Ann}(x) \cong R\oplus I$. But $R$ is semi-local; hence finitely generated projective modules cancel from direct sums. Thus Ann$(x) \cong I$ which is a principal ideal of $R$. Conversely, since $Rx \cong R/{\rm Ann}(x)$, the proof is clear.
}
\end{proof}

In the following result the spectrum of $R$ is described in case $R$ is a semi-local ring satisfies condition $(*)$ for which every maximal ideal is not cyclic.

\begin{thm} \label{prev}
Let $R$ be a semi-local ring such that satisfies condition $(*)$, Max$(R)=\{M_1, \ldots, M_t\}$, every prime ideal $P$ has the form $\bigoplus_{i\in I}Rx_i$ where $Rx_i$ is a direct sum of cyclically presented $R$-modules and $x_i\notin {\rm Nil}(R)$ for every $i\in I$ where $I$ is an index set and $M_i$'s are not cyclic. Then the following statements hold:

$(1)$ {\rm Spec}$(R) = \{M_1, \ldots, M_t, P_{k,j}\}_{k\in I, j\in \{1, \ldots , t\}}$, where $P_{k,j} = \bigoplus_{i\in I\setminus \{k\}}Rx_{i,j}$ such that $x_{i,j}\in R$.

$(2)$ $R/$Ann$(x_{k,j})$ is a principal ideal ring for every $k \in I$ and $j\in \{1, \ldots , t\}$.

\end{thm}
\begin{proof}
{$(1)$ Suppose that $R$ is a semi-local ring such that satisfies condition $(*)$, Max$(R)=\{M_1, \ldots, M_t\}$ and  every prime ideal $P$ has the form $\bigoplus_{i\in I}Rx_i$, where $Rx_i$ is a direct sum of cyclically presented $R$-modules for every $i\in I$ and $M_i$'s are not cyclic. Then for every $j\in \{1, \ldots , t\}$ we have $M_j = \bigoplus_{i\in I}Rx_{i,j}$, where $I$ is an index set, $x_{i,j}\notin {\rm Nil}(R)$ and $Rx_{i,j}$ is cyclically presented for each $i \in I$ and $j\in \{1, \ldots , t\}$. It is clear that Nil$(R) \neq M_i$ for each $i$, and so suppose that $P$ is a prime ideal of $R$ such that $P \notin \{M_1, \ldots, M_t\}$. Then there exists $j\in \{1, \ldots , t\}$ such that $P\subset M_j$. Since $x_{i,j}x_{k,j} = 0 \in P$, for every $i\neq k$ we deduce that either $x_{i,j}\in P$ or $x_{k,j}\in P$. Therefore there exists $k\in I$ such that $x_{k,j}\notin P$ and $x_{i,j}\in P$ for every $i\neq k$. This means that $\bigoplus_{i\in I\setminus \{k\}}Rx_{i,j}\subseteq P$. Since $P$ is prime, we conclude that

$$P = P\cap M_j = (P\cap Rx_{k,j})\oplus (\bigoplus_{i\in I\setminus \{k\}}Rx_{i,j}) = Px_{k,j} \oplus (\bigoplus_{i\in I\setminus \{k\}}Rx_{i,j}).$$
Therefore, $Px_{k,j} = Px^2_{k,j} = Px_{k,j}Rx_{k,j}$. But $Px_{k,j}$ is a direct sum of cyclic modules and $R$ satisfies condition $(*)$, hence $Px_{k,j} = 0$. Thus $M_j = P\oplus Rx_{k,j}$, and so Spec$(R) \subseteq \{M_1, \ldots, M_t, P_{k,j}\}_{k\in I, j\in \{1, \ldots , t\}}$, where $P_{k,j} = \bigoplus_{i\in I\setminus \{k\}}Rx_{i,j}$. Conversely, suppose that $ab\in P_{k,j}$, for some $a,b\in R$ such that $a,b\notin P_{k,j}$. Since $M_j$ is maximal, assume that $a\in M_j$, and so $a = r_kx_{k,j} + p_{k,j}$ for some $r_k\in R$ and $p_{k,j}\in P_{k,j}$. But $R$ satisfies condition $(*)$, which means that $a = cx_{k,j}^n + p_{k,j}$ for some unit element $c\in R$ and positive integer $n$. Since $ab\in P_{k,j}$, we deduce that $cbx_{k,j}^n + bp_{k,j}\in P_{k,j}$, and so $cbx_{k,j}^n \in P_{k,j}$. Therefore $cbx_{k,j}^n \in P_{k,j}\cap Rx_{k,j} = 0$. Since $R$ satisfies condition $(*)$, $b\in {\rm Ann}(x_{k,j}^n) \subseteq M_j$. Hence $b = dx_{k,j}^n + p_{2}$ for some unit element $d\in R$, positive integer $m$ and $p_2\in P_{k,j}$. Thus $ab = cdx_{k,j}^{m+n} \in P_{k,j}\cap Rx_{k,j} = 0$. This means that $x_{k,j}\in {\rm Nil}(R)$, a contradiction. Hence either $a\in P_{k,j}$ or $b\in P_{k,j}$ and thus Spec$(R) = \{M_1, \ldots, M_t, P_{k,j}\}_{k\in I, j\in \{1, \ldots , t\}}$, where $P_{k,j} = \bigoplus_{i\in I\setminus \{k\}}Rx_{i,j}$.

$(2)$ By Part $(1)$, {\rm Spec}$(R) = \{M_1, \ldots, M_t, P_{k,j}\}_{k\in I, j\in \{1, \ldots , t\}}$, where $P_{k,j} = \bigoplus_{i\in I\setminus \{k\}}Rx_{i,j}$, and so \cite[p. 8, Exercise 10]{Kaplansky} implies that  $R/P_{k,j}$ is a principal ideal ring for each $k \in I$ and $j\in \{1, \ldots , t\}$, as every prime ideal of $R/P_{k,j}$ is cyclic.
%3ashen yemken bas l maximal ideals fina nekhod ta ykoun shamele P_{k,j}
 Also, $P_{k,j}\subset$ Ann$(x_{k,j})$ for each $k \in I$ and $j\in \{1, \ldots , t\}$. Therefore, $R/$Ann$(x_{k,j})$ is a homomorphic image of $R/P_{k,j}$ and hence $R/$Ann$(x_{k,j})$ is a principal ideal ring for each $k \in I$ and $j\in \{1, \ldots , t\}$.
}
\end{proof}

The following lemma is needed for the proof of Theorem \ref{3437}.

\begin{lem} \label{2626} {\rm (See \cite[Proposition 2.7]{Heidari}).}
Let $R$ be a semi-local ring and let $M \in$ {\rm Max}$(R)$. If $M = \bigoplus_{i\in I}Rx_i = \bigoplus_{j\in J}Ry_j$ where $I,J$ are index sets and $Rx_i, Ry_j$ are nonzero cyclic $R$-modules, then $|I| = |J|$.
\end{lem}

\begin{thm} \label{3437}
Let $R$ be a semi-local ring such that satisfies condition $(*)$,  Max$(R)=\{M_1, \ldots, M_t\}$, every prime ideal $P$ has the form $\bigoplus_{i\in I}Rx_i$ where $Rx_i$ is a direct sum of cyclically presented $R$-modules and $x_i\notin {\rm Nil}(R)$ for every $i\in I$ where $I$ is an index set. Then for every maximal ideal $M_j$ of $R$ either $M_j$ is cyclic or $M_j = Rx_{1,j}\oplus Rx_{2,j}$ such that $Rx_{1,j}$ and $Rx_{2,j}$ are cyclically presented. Moreover, $R/{\rm Ann}(x_{1,j})$ and $R/{\rm Ann}(x_{2,j})$ are principal ideal rings for every $j\in \{1, \ldots , t\}$.

\end{thm}
\begin{proof}
{ Let $R$ be a semi-local ring such that satisfies condition $(*)$, Max$(R)=\{M_1, \ldots, M_t\}$ every prime ideal $P$ has the form $\bigoplus_{i\in I}Rx_i$ where $Rx_i$ is a direct sum of cyclically presented $R$-modules and $x_i\notin {\rm Nil}(R)$ for every $i\in I$ where $I$ is an index set. Suppose that $M_j$ is not cyclic for some $j\in \{1, \ldots , t\}$. Then by Theorem \ref{prev} Part $(2)$, $R/$Ann$(x_{i,j})$ is a principal ideal ring and $Rx_{1,j}$ and $Rx_{2,j}$ are cyclically presented for each $j\in \{1, \ldots , t\}$. Since $M_j$ is not cyclic, one may assume that $x = x_{1,j}$ and $y = x_{2,j}$ are nonzero. Hence, $M_j = Rx \oplus Ry \oplus L$, where $L = \bigoplus_{i\in I\setminus \{1, 2\}}Rx_{i,j}$. We claim that Ann$(x) = Ry \oplus L$. It is clear that  $Ry \oplus L\subseteq$ Ann$(x)$. Suppose that $a\in$ Ann$(x)$. Since $R$ satisfies condition $(*)$,  $a = rx + z \in$ Ann$(x)$, where $r \in R$ and $z \in Ry \oplus L$. If $rx = 0$, then Ann$(x) = Ry \oplus L$. So, suppose that $rx \neq 0$, then $rx = bx^n$ for some unit element $b$ and positive integer $n$, as $R$ satisfies condition $(*)$ and $R/$Ann$(x)$ is a principal ideal ring. This means that $0 = ax = (rx + z)x = rx^2 = bx^{n+1}$.
Hence $x^{n+1} = 0$, and thus $x \in$ Nil$(R)$ which is a contradiction. This follows that Ann$(x) = Ry \oplus L$ and so the claim is proved. Since $Rx$ is cyclically presented, Proposition \ref{2.3} follows that Ann$(x)$ is cyclic, and so Ann$(x) = R\alpha$ for some $\alpha \in R$. Therefore $M_j =  R\alpha \oplus Rx$. Thus Lemma \ref{2626} implies that $|I| = 2$. This completes the proof.
}
\end{proof}

\begin{cor} \label{2.11}
Let $R$ be a semi-local ring such that satisfies condition $(*)$, Max$(R)=\{M_1, \ldots, M_t\}$ and  every prime ideal $P$ has the form $\bigoplus_{i\in I}Rx_i$ where $Rx_i$ is a direct sum of cyclically presented $R$-modules and $x_i\notin {\rm Nil}(R)$ for every $i\in I$ where $I$ is an index set. Then for every maximal ideal $M_j$ of $R$ either $M_j$ is cyclic or $M_j = Rx_{1,j}\oplus Rx_{2,j}$ such that $Rx_{1,j}$ and $Rx_{2,j}$ are cyclically presented. Moreover,

$(1)$ If $M_j = Rx_j$ such that $x_j \notin Nil(R)$ for every $j\in \{1, \ldots, t\}$, then {\rm Spec}$(R) = \{(0), M_1, \ldots, M_t\}$.

$(2)$ If $M_j = Rx_{1,j}\oplus Rx_{2,j}$ such that $x_{1,j}, x_{2,j}\notin {\rm Nil}(R)$ for every $j\in \{1, \ldots, t\}$, then {\rm Spec}$(R) = \{Rx_{1,j}, Rx_{2,j}, M_1, \ldots, M_t\}_{j\in \{1, \ldots, t\}}$.

\end{cor}
\begin{proof}
{ $(1)$ follows from Theorem \ref{222}.

$(2)$ follows from Theorem \ref{prev}.

}
\end{proof}

Part $(2)$ of Corollary \ref{2.11} implies that $R$ is a local ring. Because, if $R$ has at
least two maximal ideals of the form $M_1 = P_{1,1}\oplus P_{2,1}$ and $M_2 = P_{1,2}\oplus P_{2,2}$
where $P_{i,j} = Rx_{i,j}$ for $1 \geq i,j \geq 2$, then $P_{1,1}P_{2,1} = 0 \subseteq P_{1,2},P_{2,2}$. So,
without loss of generality, $P_{1,1} \subseteq P_{1,2}$ and $P_{2,1} \subseteq P_{2,2}$, since $P_{1,2}\cap P_{2,2} = 0$.
This follows that $M_1 \subseteq M_2$, a contradiction.

\begin{cor}
Let $R$ be a semi-local ring such that satisfies condition $(*)$, Max$(R)=\{M_1, \ldots, M_t\}$, every prime ideal $P$ has the form $\bigoplus_{i\in I}Rx_i$ where $Rx_i$ is a direct sum of cyclically presented $R$-modules and $x_i\notin {\rm Nil}(R)$ for every $i\in I$ where $I$ is an index set. Then $R$ is Noetherian.

\end{cor}
\begin{proof}
{By Corollary \ref{2.11}, every prime ideal of $R$ is finitely generated and so Cohen Theorem (see \cite[Theorem 2]{Cohen}) implies that $R$ is Noetherian.
}
\end{proof}

Next a class of semi-local rings for which every ideal is a direct sum of cyclically presented $R$-modules is given.

\begin{thm} \label{vieww}
Let $R$ be a semi-local ring such that satisfies condition $(*)$, Max$(R)=\{M_1, \ldots, M_t\}$, $M_j = Rx_{1,j}\oplus Rx_{2,j}$ where $x_{1,j}, x_{2,j}\notin {\rm Nil}(R)$ and let $R/{\rm Ann}(x_{1,j})$ and $R/{\rm Ann}(x_{2,j})$ are principal ideal rings for every $j\in \{1, \ldots , t\}$. Then every ideal of R is a direct sum of cyclically presented R-modules.

\end{thm}
\begin{proof}
{ Let $I$ be a nonzero ideal of $R$. Then $I\subseteq M_j = Rx_{1,j}\oplus Rx_{2,j}$ for some $j\in \{1, \ldots , t\}$. If $I\subseteq Rx_{1,j}$, then $I$ is cyclic, as $Rx_{1,j} \cong R/{\rm Ann}(x_{1,j})$ is a principal ideal ring, and hence $I = Rz$ for some element $z$ of $R$. Since $R$ satisfies condition $(*)$,  $z = ax_{1,j}^n$ for some unit element $a$ and positive integer $n$. Thus $I = Rx_{1,j}^n$. It is shown that Ann$(x_{1,j}^n) = Rx_{2,j}$. It is easy to see that $Rx_{2,j} \subseteq$ Ann$(x_{1,j}^n)$. So suppose that $\alpha \in$ Ann$(x_{1,j}^n)$. Then $\alpha = rx_{1,j} + sx_{2,j}$ where $r, s \in R$, since $R$ satisfies condition $(*)$. If $rx_{1,j} \neq 0$, again by condition $(*)$,  $rx_{1,j} = bx_{1,j}^m$ for some unit element $b$ and positive integer $m$. This means that $bx_{1,j}^{m+n} = (rx_{1,j} +sx_{2,j})x_{1,j}^n = 0$. Thus, $x_{1,j}^{n+m} = 0$, a contradiction since $x_{1,j}\notin {\rm Nil}(R)$, and hence $I = Rx_{1,j}^n\cong R/{\rm Ann}(x_{1,j}^n) = R/Rx_{2,j}$ which is cyclically presented. If $I\subseteq Rx_{2,j}$, a similar argument shows that $I$ is cyclically presented. Therefore, assume that $I \nsubseteq Rx_{1,j}$ and $I \nsubseteq Rx_{2,j}$. Now, suppose that $0\neq w = rx_{1,j} + sx_{2,j}\in I$ where $r, s \in R$. By condition $(*)$,  $w = ax_{1,j}^n + bx_{2,j}^m$ for some unit elements $a,b\in R$ and positive integers $n$ and $m$, and so $ax_{1,j}^{n+1} = wx_{1,j}$ and $bx_{2,j}^{m+1} = wx_{2,j}$. Let $e$ and $t$ be the smallest positive integers such that $x_{1,j}^e \in I$ and $x_{2,j}^t\in I$. By a similar argument as above one may see that Ann$(x_{1,j}^e) = Rx_{2,j}$ and Ann$(x_{2,j}^t) = Rx_{1,j}$, which means that $Rx_{1,j}^e$ and $Rx_{2,j}^t$ are cyclically  presented $R$-modules. If  $I =Rx_{1,j}^e \oplus Rx_{2,j}^t$, then $I$ is a direct sum of cyclically presented $R$-modules. Hence, without loss of generality assume that $I \neq Rx_{1,j}^e \oplus Rx_{2,j}^t$. It is easy to see that $ Rx_{1,j}^e \oplus Rx_{2,j}^t \subset I$, and thus there exists an element $u = a_1x_{1,j}^{n_1} + b_1x_{2,j}^{m_1} \in I \setminus Rx_{1,j}^e \oplus Rx_{2,j}^t$ for some unit elements $a_1,b_1\in R$ and positive integers $n_1$ and $m_1$. We show that $n_1 = e-1$ and $m_1 = t - 1$. First, suppose that $n_1 \geq e$ and $m_1 \geq t$. Then $u \in Rx_{1,j}^e \oplus Rx_{2,j}^t$, a contradiction. If $n_1 \geq e$ and $m_1 < t$, then $a_1x_{1,j}^{n_1} \in Rx_{1,j}^e \subset I$, and so $b_1x_{2,j}^{m_1}\in I$. This means that $x_{2,j}^{m_1}\in I$, a contradiction. If $m_1 \geq t$ and $n_1 < e$, then $b_1x_{2,j}^{m_1} \in Rx_{2,j}^{t} \in I$, and so $a_1x_{1,j}^{n_1} \in I$. Therefore $x_{1,j}^{n_1}\in I$, a contradiction. Thus,   $n_1 < e$ and $m_1 < t$. If $n_1 < e-1$, then one may see that $ux_{1,j} = a_1x_{1,j}^{n_1+1} \in I$. Therefore, $x_{1,j}^{n_1+1} \in I$ which is a contradiction. Hence $n_1 = e-1$. A similar argument shows that $m_1 = t - 1$, and thus $u = a_1x_{1,j}^{e-1} + b_1x_{2,j}^{t-1}$. In the following, we show that $I = Ru$. One can easily see that $Rx_{1,j}^e \oplus Rx_{2,j}^t\subseteq Ru\subseteq I$. Conversely, suppose that $v \in I \setminus (Rx_{1,j}^e \oplus Rx_{2,j}^t)$. As above, we have $v = a_2x_{1,j}^{e-1} + b_2x_{2,j}^{t-1}$ for some unit elements $a_2,b_2\in R$. Thus, $v = a_2a_1^{-1}u + (b_2 - a_2a_1^{-1})x_{2,j}^{t-1} \in Ru\oplus Rx_{2,j}^{t-1}$. Hence $I\oplus Rx_{2,j}^{t-1}\subseteq Ru\oplus Rx_{2,j}^{t-1}$, and thus $I\oplus Rx_{2,j}^{t-1} = Ru\oplus Rx_{2,j}^{t-1}$. But $R$ is semi-local; hence finitely generated projective modules cancel from direct sums. Therefore $I = Ru$. It is enough to prove that $Ru$ is cyclically presented. Clearly, $Rx_{1,j}\cap{\rm Ann}(Rx_{1,j}) = 0$ and $Rx_{2,j}\cap {\rm Ann}(Rx_{2,j}) = 0$. We claim that Ann$(u) = 0$. Suppose to the contrary, $0\neq w \in{\rm Ann}(u)$. Then $w = a_3x_{1,j}^{n_3} + b_3x_{2,j}^{m_3}$ for some unit elements $a_3,b_3\in R$ and positive integers $n_3$ and $m_3$ since Ann$(u)\subseteq{\rm Ann}(Ru) = {\rm Ann}(Rx_{1,j})\cap{\rm Ann}(Rx_{2,j})\subseteq M_j$. This means that $wa_1x_{1,j}^{n_1} = a_1a_3x_{1,j}^{n_1+n_3}\in Rx_{1,j}\cap{\rm Ann}(u) \subseteq Rx_{1,j}\cap{\rm Ann}(Rx_{1,j})\cap{\rm Ann}(Rx_{2,j}) = 0$. By a similar argument as above one may see that $wb_1x_{2,j}^{m_1} = b_1b_3x_{2,j}^{m_1+m_3}\in Rx_{2,j}\cap{\rm Ann}(u) \subseteq Rx_{2,j}\cap{\rm Ann}(Rx_{1,j})\cap{\rm Ann}(Rx_{2,j}) = 0$. Thus $x_{1,j}, x_{2,j}\in$ Nil$(R)$, a contradiction. Therefore, Ann$(u) = 0$, and hence $I = Ru \cong R/(0)$ is cyclically presented.}
\end{proof}

Consequently, in view of the proof of Theorem \ref{vieww}, we state the following corollary.

\begin{cor}
Let $R$ be a semi-local ring such that satisfies condition $(*)$, Max$(R)=\{M_1, \ldots, M_t\}$, $M_j = Rx_{1,j}\oplus Rx_{2,j}$ where $x_{1,j}, x_{2,j}\notin {\rm Nil}(R)$ and let $R/{\rm Ann}(x_{1,j})$ and $R/{\rm Ann}(x_{2,j})$ are principal ideal rings for every $j\in \{1, \ldots , t\}$. Then every ideal of $R$ is one of the following forms $R$, $Rx_{1,j}^n$ for some positive integer $n$ and $j\in \{1, \ldots , t\}$, $Rx_{2,j}^m$ for some positive integer $m$ and $j\in \{1, \ldots , t\}$, and $Rx_{1,j}\oplus Rx_{2,j}$ for some $j\in \{1, \ldots , t\}$.

\end{cor}

We are now ready for the main result of this paper. In fact, we give two gauges to check whether every ideal of a semi-local ring $R$ is a direct sum of cyclically presented modules, it suffices to test only the prime ideals or structures of the maximal ideals of $R$.

\begin{thm} \label{2.112}
Let $R$ be a semi-local ring such that satisfies condition $(*)$, Max$(R)=\{M_1, \ldots, M_t\}$ and if $M_j = \bigoplus_{i\in I}Rx_{i,j}$, then $x_{i,j} \notin {\rm Nil}(R)$ for every $i\in I$. Then the following statements are equivalent:

$(1)$ Every ideal of R is a direct sum of cyclically presented $R$-modules.

$(2)$ Every ideal of $R$ is a direct summand of a direct sum of cyclically presented $R$-modules.

$(3)$ Every prime ideal of $R$ is a direct sum of cyclically presented R-modules.

$(4)$ Every prime ideal of $R$ is a direct summand of a direct sum of cyclically presented R-modules.

$(5)$ Either $R$ is a principal ideal ring or $M_j = Rx_{1,j}\oplus Rx_{2,j}$, where $R/{\rm Ann}(x_{1,j})$ and $R/{\rm Ann}(x_{2,j})$ are principal ideal rings for every $j\in \{1, \ldots , t\}$.

\end{thm}

\begin{proof}
{$(5) \Rightarrow (1)$ It follows from Theorem \ref{vieww}.

$(3) \Rightarrow (5)$ It follows from Theorem \ref{3437}, Corollary \ref{2.11} and \cite[p. 8, Exercise 10]{Kaplansky}.

$(1) \Leftrightarrow (2)$ and $(3) \Leftrightarrow (4)$ follow from \cite[Proposition 3]{Warfield}.
}
\end{proof}

We close this paper with the following example.

\begin{example} {\rm Let $F$ be a field and $R := R_1 \oplus R_2$ such that $R_1 = F[[X, Y ]]/{\rm<}XY{\rm>}$ and $R_2$ is a field where $F[[X, K]]$ is a formal power series ring. It is easy to see that $R_1$ and $R_2$ are local rings with maximal ideals $M_1 = R_1x \oplus R_1y$ and $M_2 = \{0\}$ respectively, where $x = X + {\rm<}XY{\rm>}$, $y = Y + {\rm<}XY{\rm>}$ and $x, y\notin$ Nil$(R_1)$. Also, Max$(R) = \{M_1\oplus R_2, R_1\oplus M_2\}$. First we show that $R$ satisfies condition $(*)$. Suppose that $a\in R_1x$ such that $a\neq 0$. Since $R_1$ is a local ring and $R_1/{\rm Ann}(x)$ is a principal ideal ring by \cite[Example 2.17]{Moradzadeh}, it follows from \cite[Proposition 3.4]{Behboodi2} that $a = bx^n$ for some unit $b$ and positive integer $n$. Hence the axiom $(1)$ in condition $(*)$ is satisfied, and so if $M_i = Rx_i \oplus L$ for some ideal $L$ of $R$ and some nonunit $x_i \in R$ such that $R/{\rm Ann}(x_i)$ is a principal ideal ring, then every nonzero element of $Rx_i$ is of the form $ax^n$ for some unit $a$ and positive integer $n$. Now, since $x,y$ are nonunit in $R_1$ and $x, y\notin$ Nil$(R_1)$, we conclude that Ann$(x_i^{\alpha_i}) \subset M_1$ for every $x_i\in \{x, y\}$ and positive integer $\alpha_i$. This means that Ann$(x_i^{\alpha_i}) \subseteq M_1\oplus R_2$ for every $x_i\in \{x, y\}$ and positive integer $\alpha_i$. It remains only for us show that for every finitely generated module $N$ over the commutative ring $R$, if $I$ is an ideal of $R$ such that $I \subseteq M_j$ for some  $M_j \in$ Max$(R) = \{M_1\oplus R_2, R_1\oplus M_2\}$ and $N = IN$, then $N = 0$. Suppose contrary that $N\neq 0$ and look for a contradiction. Assume that $\{a_1, \ldots , a_n\}$ is a minimal generating set for the module $N$ over the commutative ring $R$ and $I \subseteq M_1\oplus R_2$ such that $N = IN$. Since $a_1 \in IN$, we have $a_1 = \sum_{i=1}^na_ib_i$ for some $b_1, \ldots , b_n\in I$, and so $(1 - b_1)a_1 = a_2b_2 + \cdots + a_nb_n$. Suppose that $b_1 = m + r_2 \in M_1\oplus R_2$ for some $m\in M_1$ and $r_2\in R_2$. If $r_2\in M_2$, then $b_1\in M_1\oplus M_2\subseteq R_1 \oplus \{0\}$, and so $1 - b_1$ is a unit of $R$ by \cite[Lemma 3.17]{sha}. Thus $N$ is generated by $\{a_2, \ldots, a_n\}$, a contradiction. If $r_2\notin M_2$, then $r_2$ is a unit of $R_2$. This means that $(1-m_1)a_1 -r_2a_1 = \sum_{i=2}^na_ib_i$. Hence $(r_2 -r_2^2)a_1 = \sum_{i=2}^na_ib_ir_2$. Since $R_2$ is a field we conclude that $(r_2 -r_2^2)$ is unit, and hence $N$ is generated by $\{a_2, \ldots, a_n\}$, a contradiction. Thus $N = 0$. In the other hand, it follows from Theorem \ref{2.112} that every ideal of $R$ is a direct sum of cyclically presented modules.
}
\end{example}

%%%%%%%%%%%%%%%%%%%%%%%%%%%%%%%%%%%%%%%%%%%%%%%%%%%%%%%%%%%%%%%%%%%%%%%%%%%%%%%%%%%%%%%%%%%%%%%%%%%%%%%%%%%%%%%%%%%%%%%%%%%%%%%%%%%%%%%%%%%%%%%%%%%%%%%%%%%%%%%%%%%%%%%%%%%%%%%%%%%%%%%%%%

%\noindent{\bf Acknowledgements.} The authors express their deep gratitude to the referee for his/her meticulous reading and valuable
%suggestions which have definitely improved the paper.
%%%%%%%%%%%%%%%%%%%%%%%%%%%%%%%%%%%%%%%%%%%%%%%%%%%%%%%%%%%%%%%%%%%%%%%%%%%%%%%%%%%%%%%%%%%%%%%%%%%%%%%%%%%%%%%%%%%%%%%%%%%%%%%%%%%%%%%%%%%%%

%%%%%%%%%%%%%%%%%%%%%%%%%%%%%%%%%%%%%%%%%%%%%%%%%%%%%%%%%%%%%%%%%%%%%%%%%%%%%%%%%%%%%%%%%%%%%%%%%%%%%%%%%%%%%%%%%%%%%%%%%%%%%%%%%%%%%%%%%%%%%%%%%%%%%

\end{document}